\documentclass[12pt]{amsart}
\usepackage[left=.75in,right=.75in,top=.75in,bottom=.75in]{geometry}
\usepackage{setspace}
\newtheorem{theorem}{Theorem}[section]

\newtheorem{lemma}[theorem]{Lemma}

\theoremstyle{definition}

\theoremstyle{remark}
\newtheorem{remark}[theorem]{Remark}
\numberwithin{equation}{section}

\begin{document}

\title[stability of the $p$-affine isoperimetric inequality]
 {On the stability of the $p$-affine isoperimetric inequality}
\author[M.N. Ivaki]{Mohammad N. Ivaki}
\address{Department of Mathematics and Statistics,
  Concordia University, Montreal, QC, Canada, H3G 1M8}
\curraddr{}
\email{mivaki@mathstat.concordia.ca}
\date{\today}

\dedicatory{}
\doublespace
\begin{abstract}
Employing the affine normal flow, we prove a stability version of the $p$-affine isoperimetric inequality for $p\geq1$ in $\mathbb{R}^2$ in the class of origin-symmetric convex bodies. That is, if $K$ is an origin-symmetric convex body in $\mathbb{R}^2$ such that it has area $\pi$ and its $p$-affine perimeter is close enough to the one of an ellipse with the same area, then, after applying a special linear transformation, $K$ is close to an ellipse in the Hausdorff distance.
\end{abstract}

\maketitle
\section{Introduction}

The setting for this paper is the $n$-dimensional Euclidean space, $\mathbb{R}^n.$ A compact convex subset of $\mathbb{R}^n$ with non-empty interior is called a \emph{convex body}.

Let $\Phi$ be a real valued function on convex bodies. Given a geometric inequality $\Phi(K)\geq 0$, for every convex body $K$ and with the equality case obtained only for a certain family of convex bodies, denoted by $\mathcal{F}$, a \emph{stability} version of $\Phi$ concerns the following question. Find a positive constant $\varepsilon_0$, and a positive function $f$, such that the following holds:
If for some $0<\varepsilon\leq\varepsilon_0$ we have
$$\Phi(K)\leq \varepsilon,$$
then there exists a convex body in $\mathcal{F}$, denoted by $L$, such that
$$d(K,L)\leq f(\varepsilon),$$
where $d(\cdot,\cdot)$ is an appropriate norm in the context of the geometric inequality. Here $f$ obeys the rule $\lim\limits_{\varepsilon\to0}f(\varepsilon)=0$
(see the beautiful survey of H. Groemer \cite{Groemer}).

Versions of stability have been investigated for several important inequalities, including a stability version of the Brunn-Minkowski inequality due to V.I. Diskant \cite{Dis}, stability of the Rogers-Shephard inequality by K.J. B\"{o}r\"{o}czky \cite{B1}, stability of the Blaschke-Santal\'{o} inequality and the affine isoperimetric inequality in $\mathbb{R}^n$ for $n\geq3$ by K.J. B\"{o}r\"{o}czky \cite{B}, stability of the reverse Blaschke-Santal\'{o} inequality by K.J. B\"{o}r\"{o}czky and D. Hug \cite{BH}, stability of the Pr\'{e}kopa-Leindler inequality by K. Ball and K.J. B\"{o}r\"{o}czky \cite{BB}, stability of a volume ratio by D. Hug and R. Schneider \cite{HS}, and more recently stability of the functional forms of the Blaschke-Santal\'{o} inequality by F. Barthe, K.J. B\"{o}r\"{o}czky and M. Fradelizi \cite{BBF}. Our aim in this paper is to prove a stability version of the $p$-affine isoperimetric inequality for $p\geq1$ in $\mathbb{R}^2$ in the class of origin-symmetric convex bodies.

In his seminal work \cite{Lutwak1,Lutwak2}, E. Lutwak extended the Brunn-Minkowski theory to the Brunn-Minkowski-Firey theory yielding impressive new results in convex geometry \cite{LYZ1,LYZ3,LYZ4,LYZ5}, stochastic geometry \cite{Gr1,Gr2}, differential geometry and differential equations \cite{CL,ivaki,LO,TW1,TW2,TW3, TW4}. One of these developments was extension of the notion of the affine surface area to $p$-affine surface areas for $p>1$. Subsequently, the notion of $p$-affine surface areas for $0<p<1$ has been introduced by D. Hug \cite{Hug}, for $-n<p<0$ by M. Meyer and E. Werner \cite{MW}, and for all $p\neq -n$ by C. Sch\"{u}tt and E. Werner in \cite{SW}. Later, in \cite{Ludwig,Ludwig1} it was observed by M. Ludwig that $p$-affine surface areas, $p\neq -n$, belong to a larger family, called $\phi$-affine surface areas. For $p\geq1$, the $p$-affine surface area of a convex body is related to the volume of the convex body by the $p$-affine isoperimetric inequality. For $p=1$, this is the well-known affine isoperimetric inequality due to W. Blaschke with the equality case characterized in the class of convex bodies with $\mathcal{C}^2$ boundary \cite{Bl}. The characterization of the equality in general is due to C.M. Petty \cite{Petty}. The $p$-affine isoperimetric inequality, for $p>1$, was proved by E. Lutwak \cite{Lutwak2}, including characterizing the equality case. The equality in the $p$-affine isoperimetric inequality is achieved only for ellipsoids centered at the origin. The $p$-affine isoperimetric inequality, for $p<1,~ p\neq -n$ was proved by E. Werner and D. Ye \cite{WY}. Their inequalities for $p<-n$ depend on the constant arising from the inverse Blaschke-Santal\'{o} inequality.

A stability version of the affine isoperimetric inequality, $p=1$, was presented by K.J. B\"{o}r\"{o}czky in $\mathbb{R}^n$ for $n\geq3$, \cite{B}. He proved that if $K$ be a convex body in $\mathbb{R}^{n}$ such that its affine surface is $\varepsilon$-close to the one of an ellipsoid, for a fixed $\varepsilon\in(0,\frac{1}{2})$, then $K$ is close to the unit ball in the Banach-Mazur distance. Here, close is an approximation of order $\varepsilon^{\frac{1}{6n}}|\log \varepsilon|^{\frac{1}{6n}}.$ Later in \cite{BB2}, the order of approximation was improved to $\varepsilon^{\frac{1}{3(n+1)}}|\log\varepsilon|^{\frac{4}{3(n+1)}}$. The case $n=2$ was not addressed either in \cite{BB2} or in \cite{B}. Later in \cite{BM}, K.J. B\"{o}r\"{o}czky and E. Makai proved a stability version of the Blaschke-Santal\'{o} inequality from which stability for the $p$-affine isoperimetric inequality
follows easily.
In this paper, using a different method, we prove a version of stability of the $p$-affine isoperimetric inequality for $p\geq1$ in the class of origin-symmetric convex bodies in $\mathbb{R}^2.$ The technique presented here to deal with stability, is new as it approaches the problem from the perspective of geometric flows and ODEs. However, the interaction between convex geometry and geometric flows is not new. There are several important contributions of geometric flows to convex geometry, for example, a proof of the affine isoperimetric inequality by B. Andrews using the affine normal flow \cite{BA2}, obtaining the necessary and sufficient conditions for the existence of a solution to the discrete $L_0$-Minkowski problem using crystalline curvature flow by A. Stancu \cite{S0,S5,S3} and independently by B. Andrews \cite{BA5}, and a proof of the $p$-affine isoperimetric inequality in the class of origin-symmetric convex bodies in $\mathbb{R}^2$ using the affine normal \cite{ivaki}. See \cite{S1,S2,S,S4} for more applications of flows, in particular, a newly defined family of centro-affine $p$-flows and their applications to centro-affine differential geometry by A. Stancu \cite{S,S4}.

Let $K$ be a convex body. The support function of $K$, denoted by $s_K$, is defined as
\begin{align*}
s_K&:\mathbb{S}^{n-1}\to\mathbb{R}\\
 s_{K}(z)&=\max_{y\in \partial K}\langle z,y\rangle,
\end{align*}
where $\langle z,y\rangle$ denotes the standard inner product of $z$ and $y$.

 Fix a compact, smooth, strictly convex hypersurface $\mathcal{M}$. Let $K$ be a strictly convex body, having the origin in its interior such that its boundary, denoted by $\partial K$, is smoothly embedded in $\mathbb{R}^{n}$ by
$$x_K:\mathcal{M}\to\mathbb{R}^{n}.$$
 Therefore, we have $x_K(\mathcal{M})=\partial K$. For simplicity, in the remainder of the paper, we will identify the domain of an embedding with its image. Let $\mathbf{N}_K(x)$ be the outward unit normal vector of $K$ for every $x\in \partial K$. The support function of $K$ has the following simple form
 $$s_{K}(z):= \langle \mathbf{N}_K^{-1}(z), z \rangle,$$
for each $z\in\mathbb{S}^{n-1}$, where $\mathbf{N}_K^{-1}:\mathbb{S}^{n-1}\to \partial K$ is the inverse of the Gauss map $\mathbf{N}_K$. We denote the standard metric on $\mathbb{S}^{n-1}$ by $\bar{g}_{ij}$ and the standard Levi-Civita connection of $\mathbb{S}^{n-1}$ by $\bar{\nabla}$.
We denote the Gauss curvature of $\partial K$ by $\mathcal{K}$ and remark that, as a function on $\partial K$, it is related to the support function of the convex body by $$\frac{1}{\mathcal{K}\circ\mathbf{N}_K^{-1}}:=\det_{\bar{g}}(\bar{\nabla}_i\bar{\nabla}_js+\bar{g}_{ij}s).$$

Furthermore, the affine support function of $K$, denoted by $\sigma$, as a function on $\partial K$ is defined by
$$\sigma(x)=\frac{\langle x,\mathbf{N}_K(x)\rangle}{{\mathcal{K}^{1/(n+1)}}(x)},$$
for all $x\in \partial K.$

For $p\geq 1$, the $p$-affine surface area of $K$ is defined by
$$\Omega_p(K)=\int_{\partial K}\frac{\mathcal{K}^{\frac{p}{n+p}}(x)}{\langle x,\mathbf{N}_K(x)\rangle^{\frac{n(p-1)}{n+p}}}d\mu_{\partial K}(x),$$
where $\mu_{\partial K}$ is the usual surface area measure on $\partial K.$ The $p$-affine surface area of a convex body is bounded by the volume via the $p$-affine isoperimetric inequality. If the centroid of $K$ is at the origin then
$$\left(\frac{\Omega_p^{n+p}(K)}{n^{n+p}V^{n-p}(K)}\right)^{\frac{1}{p}}\leq\omega_{n}^{2},$$
with the equality case only for ellipsoids centered at the origin.
Here, $V(K)$ is the volume of $K$ defined by $V(K)=\frac{1}{n}\int_{\partial K}\langle x,\mathbf{N}_K(x)\rangle d\mu_{\partial K}(x),$ and $\omega_{n}$ is the volume of the unit ball of $\mathbb{R}^n$.
We call the quantity $\left(\frac{\Omega_p^{n+p}(K)}{n^{n+p}V^{n-p}(K)}\right)^{1/p},$ the $p$-affine isoperimetric ratio of $K.$ In $\mathbb{R}^2$, it is more appropriate to use the notation $A(K)$ instead of $V(K)$ for the area of $K.$

We mention here that thanks by a theorem of A.D. Alexandrov (see P.M. Gruber \cite{Gr2}, page 22), the boundary of a convex body is twice differentiable in a generalized sense almost everywhere with respect to its Hausdorff measure. Furthermore, the Gauss map is also defined in a generalized sense almost everywhere with respect to the Hausdorff measure of the boundary of a convex body. Therefore, generalized notions of Gauss curvature and Gauss map are available for convex bodies which are not necessarily smooth. This in turn implies that the formula above of the $p$-affine surface area is still valid for all convex bodies.

Let $K$ and $L$ be two origin-symmetric convex bodies in $\mathbb{R}^{n}$ with respective support functions $s_K$ and $s_L$. Then the Hausdorff distance between $K$ and $L$ is defined by
$$d_{\mathcal{H}}(K,L)=\max_{\mathbb{S}^{n-1}}|s_K-s_L|.$$

In what follows, we mainly work in $\mathbb{R}^2.$ Without loss of generality, using John's lemma \cite{J}, if necessary, we can assume that $c_1\leq s_{K}\leq c_2$ for universal constants $c_1$ and $c_2$, depending only on $A(K)$.

 \begin{theorem}[Main Theorem]Let $p\geq1.$ There exists an $\varepsilon_p>0$, depending on $p$, such that the following holds. Let $K$ be an origin-symmetric convex body with area $\pi$. If for an $\varepsilon$, $0<\varepsilon<\varepsilon_p$
 $$\left(\frac{\Omega_p^{2+p}(K)}{2^{2+p}A^{2-p}(K)}\right)^{\frac{1}{p}}>\pi^2(1-\varepsilon),$$
then there exist a disk $\mathcal{D}$, an ellipse $\mathcal{E}$ and a special linear transformation $T$ such that
$$\mathcal{E}\subseteq TK\subseteq \left(1+\left(\frac{4}{3}\right)^{\frac{3}{4}}\frac{2c_2}{c_1^2}\varepsilon^{\frac{3}{10}}\right)\mathcal{D},$$
and
$$d_{\mathcal{H}}\left(\mathcal{E},\left(1+\left(\frac{4}{3}\right)^{\frac{3}{4}}\frac{2c_2}{c_1^2}\varepsilon^{\frac{3}{10}}\right)
\mathcal{D}\right)
<C_p\varepsilon^{\frac{3}{10}},$$
for a universal constant $C_p.$\\
In particular,
$$d_{\mathcal{H}}\left(TK,\mathcal{E}\right)
<C_p\varepsilon^{\frac{3}{10}}.$$
\end{theorem}
To prove this theorem we will implement the affine normal flow on curves. We only use results on the short time behavior of this flow. Fix a smooth, origin-symmetric, strictly convex curve $\mathcal{M}$.
Let $K$ be a compact, origin-symmetric, strictly convex body, smoothly embedded in $\mathbb{R}^2$. We denote the space of such convex bodies by $\mathcal{K}_{sym}$. Let
 $$x_K:\mathcal{M}\to\mathbb{R}^2,$$
be a smooth embedding of $\partial K$, the boundary of $K\in \mathcal{K}_{sym}$.
 We denote the curvature of $\partial K$ by $\kappa$, as a function on $ \mathcal{M}$, the curvature is related to the support function by
 $$\frac{1}{\kappa(\mathbf{n}_K^{-1}(z))}:=s_{\theta\theta}(z)+s(z),$$
 where $\theta$ is the angle parameter on $\mathbb{S}^1$ identified with $z$, and $\mathbf{n}_K^{-1}$ is the inverse of the Gauss map of $K$ denoted by $\mathbf{n}_K.$
Let $K_0:=K\in \mathcal{K}_{sym}$. We consider a family $\{K_t\}\in \mathcal{K}_{sym}$, and their associated smooth embeddings $x:\mathcal{M}\times[0,T)\to \mathbb{R}^2$, which are evolving according to the affine normal flow, namely,
 \begin{equation}\label{e: flow0}
 \partial_{t}x(\cdot,t):=-\kappa^{\frac{1}{3}}(\cdot,t)\, \mathbf{n}_{K_t}(\cdot),~~
 x(\cdot,0)=x_{K_0}(\cdot),~~ x(\cdot ,t)=x_{K_t}(\cdot).
 \end{equation}
 Note that at each time $t$ we have $x(\mathcal{M},t)=\partial K_t.$\\
The well-known affine normal flow was addressed by G. Sapiro and A. Tannenbaum \cite{ST} and by B. Andrews in more generality \cite{BA2,BA3}. Andrews investigated the affine normal flow of compact hypersurfaces in any dimension and showed that the volume preserving flow evolves any convex initial bounded open set, not necessarily smooth, exponentially fast, in the $\mathcal{C}^{\infty}$ topology, to an ellipsoid, \cite{BA3}.

We point out here that we need only to prove the main theorem for smooth $K$. The reason is the instantaneous smoothing property of the affine normal flow \cite{BA3} and monotonicity of the $p$-affine isoperimetric inequality along the affine normal flow \cite{ivaki}.

\section{Stability of the $p$-affine isoperimetric inequality}
Throughout this section we assume that $K_0=K$ is smooth and $A(K)=\pi$.
\subsection{Preliminaries}
We list several lemmas and a theorem necessary for our proof of the main theorem.
\begin{lemma}[Containment Principle]\cite{S} \label{prop: containment principle}
Let $K^{in}$ and $K^{out}$ be two convex
bodies in $\mathcal{K}_{sym}$ such that $K^{in}\subset K^{out}$, then $K^{in}_t\subseteq K^{out}_t$ for as long as
the solutions $K^{in}_t$ and $K^{out}_t$ of (\ref{e: flow0}) $($with given initial data
$K^{in}_0=K^{in}$, $K^{out}_0=K^{out}$$)$ exist
in $\mathcal{K}_{sym}.$
\end{lemma}
\begin{lemma}[Evolution equation of the area]\cite{ivaki}\label{lem: ev of A}
As $\{K_t\}$ evolve by evolution equation (\ref{e: flow0}), $A(K_t)$  evolves by $\frac{d}{dt}A(K_t)=-\Omega_1(K_t).$
In particular, $A(K_t)$ is decreasing.
\end{lemma}
We recall the following affine isoperimetric type inequalities involving time derivative of $\Omega_p$ from Lemma 6.1 in \cite{ivaki} along the affine normal flow.
 \begin{lemma}[$\Omega_p$ along the affine normal flow]\cite{ivaki}\label{lem: controlling derivative of $l$-affine length along affine normal flow}
As $\{K_t\}$ evolve by evolution equation (\ref{e: flow0}), the following affine isoperimetric inequalities hold. \\
If $1\leq p\leq 2$, then
\begin{align*}
\frac{d}{dt}\Omega_p(K_t)\geq \frac{p-2}{p+2}\frac{\Omega_p(K_t)\Omega_1(K_t)}{A(K_t)}+\frac{2(p-1)(4p^2+3p+2)}{(p+2)^3}\int_{\partial K_t}\sigma^{-1-\frac{3p}{p+2}}\sigma_{\mathfrak{s}}^2d\mathfrak{s},
\end{align*}
while, if $p\geq2$, we then have
\begin{align*}
\frac{d}{dt}\Omega_p(K_t)\geq\frac{p-2}{p+2}\frac{\Omega_p(K_t)\Omega_1(K_t)}{A(K_t)}
+\frac{6p}{(p+2)^2}\int_{\partial K_t}\sigma^{-1-\frac{3p}{p+2}}\sigma_{\mathfrak{s}}^2d\mathfrak{s}.
\end{align*}
Here, $\mathfrak{s}$ is the affine arc-length of the evolving boundary curve $\partial K_t.$ We point out here that for a curve $\partial K$
$$d\mathfrak{s}=\kappa^{\frac{1}{3}}de,$$ where $e$ is the Euclidean arc-length of $\partial K.$
 \end{lemma}
 \begin{remark}\label{re: relation affine support and area}
The affine support function is constant for an origin-centered ellipse and the relation between its value
and the area of the ellipse is as follows. For an ellipse $\mathcal{E}$, denote its constant affine support function by $\sigma_{\mathcal{E}}$. We have
$$\sigma_{\mathcal{E}}=\left(\frac{A(\mathcal{E})}{\pi}\right)^{2/3}.$$
 \end{remark}
\begin{lemma}[Stability of the affine support function]\cite{ivaki}\label{lem: ellipsoid app} Suppose that $K$ is a convex body in $\mathcal{K}_{sym}$. If $m\le\sigma\le M$ for some positive numbers $m$ and $M$, then there exist two ellipses $\mathcal{E}_{in}$ and $\mathcal{E}_{out}$ such that $\mathcal{E}_{in}\subseteq K\subseteq \mathcal{E}_{out}$ and
$$\left(\frac{A(\mathcal{E}_{in})}{\pi}\right)^{2/3}=m,~~~ \left(\frac{A(\mathcal{E}_{out})}{\pi}\right)^{2/3}=M .$$
\end{lemma}
\begin{lemma}\label{lem: minmax of affine}
Let $K$ be an origin-symmetric, smooth convex body with area $\pi$. Then
$$\min_{\partial K} \sigma\leq 1\leq\max_{\partial K}\sigma.$$
\end{lemma}
\begin{proof}
The claim follows from Lemma \ref{lem: ellipsoid app}: If $\min\limits_{\mathbb{S}^1} \sigma>1$, then there is an ellipse $\mathcal{E}_{in}$ which is contained in $K$ and satisfies
$\left(\frac{A(\mathcal{E}_{in})}{\pi}\right)^{2/3}>1.$  This implies that $A(\mathcal{E}_{in})>\pi.$ Similarly, if $\max\limits_{\mathbb{S}^1} \sigma<1$, then there is an ellipse $\mathcal{E}_{out}$ which contains $K$ and has the area $A(\mathcal{E}_{out})<\pi.$ In both cases we reach to a contradiction as the area of $K$ is $\pi.$
\end{proof}
We state the following important Theorem 5 from \cite{BA3}. We denote an origin-centered disk of radius $r>0$ by $B_r.$
\begin{theorem}[Controlling Hausdorff distance I]\label{thm: Ben}\cite{BA3}
Let $\{K_t\}$ be a smooth, strictly convex solution of the evolution equation (\ref{e: flow0}). Then
$$s(z,t)\ge s(z,0)-\left(\frac{4}{3}\right)^{\frac{3}{4}}\frac{c_2}{c_1}t^{\frac{3}{4}},$$
for $t\in\left(0,\frac{3}{4}c_1^{\frac{4}{3}}\right).$
In particular,
$$K\subseteq K_t+ \left(\frac{4}{3}\right)^{\frac{3}{4}}\frac{c_2}{c_1}t^{\frac{3}{4}}B_1$$
for $t\in\left(0,\frac{3}{4}c_1^{\frac{4}{3}}\right).$
\end{theorem}
Let $\mathcal{E}$ be an ellipse. We denote its semi-minor and semi-major axes by $a_{\mathcal{E}}$ and $b_\mathcal{E}$, respectively. We also need the following simple lemma.
\begin{lemma}[Controlling Hausdorff distance II]\label{e: ellipsoid app2}
Let $\mathcal{E}$ be an ellipse centered at the origin of the plane such that $\mathcal{E}\subseteq B_R$. Then we have
\begin{equation*}
d_{\mathcal{H}}(\mathcal{E},B_R)\leq \frac{A(B_R)-A(\mathcal{E})}{\pi\left(\frac{A(B_R)}{\pi}\right)^{\frac{1}{2}}}.
\end{equation*}
\end{lemma}
\begin{proof}
We have
\begin{align*}
d_{\mathcal{H}}(\mathcal{E},B_R)&\leq R-a_{\mathcal{E}}\\
&=\left(\frac{A(B_R)}{\pi}\right)^{\frac{1}{2}}-\frac{A(\mathcal{E})}{\pi b_{\mathcal{E}}}\\
&\leq\left(\frac{A(B_R)}{\pi}\right)^{\frac{1}{2}}-\frac{A(\mathcal{E})}{\pi R}=\frac{A(B_R)-A(\mathcal{E})}{\pi\left(\frac{A(B_R)}{\pi}\right)^{\frac{1}{2}}}.
\end{align*}
The proof is complete.
\end{proof}

\subsection{Proof of the main theorem}
In this section we present a proof of the stability of the $p$-affine isoperimetric inequality.
\begin{proof}
Let $p>1$ and $0<\varepsilon_p<\frac{1}{2}$. The upper bound on $\varepsilon_p$ will be determined later at the end of this section. Assume that
\begin{equation}\label{ie: assum}
\left(\frac{\Omega_p^{2+p}(K)}{2^{2+p}A^{2-p}(K)}\right)^{\frac{1}{p}}>\pi^2(1-\varepsilon_p).
\end{equation}
Then from Lemma \ref{lem: controlling derivative of $l$-affine length along affine normal flow} and Lemma \ref{lem: ev of A} it follows that
\begin{align}\label{e: derv of omega}
\frac{d}{dt}\left(\frac{\Omega_p^{2+p}(K_t)}{A^{2-p}(K_t)}\right)^{\frac{1}{p}}&=
\frac{1}{p}\left(\frac{\Omega_p^{2+p}(K_t)}{A^{2-p}(K_t)}\right)^{\frac{1}{p}-1}\frac{d}{dt}\left(\frac{\Omega_p^{2+p}(K_t)}{A^{2-p}(K_t)}\right)\nonumber\\
&\geq d_p\left(\frac{\Omega_p^{2+p}(K_t)}{A^{2-p}(K_t)}\right)^{\frac{1}{p}-1}\frac{\Omega_p^{p+1}(K_t)}{A^{2-p}(K_t)}\int_{\partial K_t}
\left(\sigma^{\frac{1}{2}-\frac{3p}{2(p+2)}}\right)_{\mathfrak{s}}^2d\mathfrak{s}\nonumber\\
&=\frac{d_p}{\Omega_p(K_t)}\left(\frac{\Omega_p^{2+p}(K_t)}{A^{2-p}(K_t)}\right)^{\frac{1}{p}}\int_{\partial K_t}
\left(\sigma^{\frac{1}{2}-\frac{3p}{2(p+2)}}\right)_{\mathfrak{s}}^2d\mathfrak{s}
\end{align}
where $d_p$ is defined as follows
$$d_p:=\left\{
         \begin{array}{ll}
           \frac{2(4p^2+3p+2)}{p(p-1)}, & \hbox{if}~ 1<p\leq 2, \\
           \frac{6(p+2)}{(p-1)^2}, & \hbox{if}~ p\geq2.
         \end{array}
       \right.$$

We integrate both sides of the inequality (\ref{e: derv of omega}) on the time interval $[0,\delta]$ with respect to $dt.$
\begin{align*}
\int_0^{\delta}\frac{d}{dt}\left(\frac{\Omega_p^{2+p}(K_t)}{A^{2-p}(K_t)}\right)^{\frac{1}{p}}dt
&\geq \int_0^{\delta}\frac{d_p}{\Omega_p(K_t)}\left(\frac{\Omega_p^{2+p}(K_t)}{A^{2-p}(K_t)}\right)^{\frac{1}{p}}\int_{\partial K_t}
\left(\sigma^{\frac{1}{2}-\frac{3p}{2(p+2)}}\right)_{\mathfrak{s}}^2d\mathfrak{s}dt\\
&\geq \int_0^{\delta}\min_{t\in[0,\delta]}\left(\frac{d_p}{\Omega_p(K_t)}\left(\frac{\Omega_p^{2+p}(K_t)}{A^{2-p}(K_t)}\right)^{\frac{1}{p}}\int_{\partial K_t}
\left(\sigma^{\frac{1}{2}-\frac{3p}{2(p+2)}}\right)_{\mathfrak{s}}^2d\mathfrak{s}\right)dt\\
&=\frac{d_p\delta}{\Omega_p(K_{t_{\ast}})}\left(\frac{\Omega_p^{2+p}(K_{t_{\ast}})}{A^{2-p}(K_{t_{\ast}})}\right)^{\frac{1}{p}}\int_{\partial K_{t_{\ast}}}
\left(\sigma^{\frac{1}{2}-\frac{3p}{2(p+2)}}\right)_{\mathfrak{s}}^2d\mathfrak{s}
\end{align*}
where $t_{\ast}$ is the time such that $\min\limits_{t\in[0,\delta]}\frac{1}{\Omega_p(K_t)}\left(\frac{\Omega_p^{2+p}(K_t)}{A^{2-p}(K_t)}\right)^{\frac{1}{p}}\int_{\partial K_t}
\left(\sigma^{\frac{1}{2}-\frac{3p}{2(p+2)}}\right)_{\mathfrak{s}}^2d\mathfrak{s}$ is achieved.
Therefore, using the H\"{o}lder inequality we find
$$2^{\frac{p+2}{p}}\pi^2\varepsilon_p\geq\frac{d_p\delta}{\Omega_1(K_{t_{\ast}})\Omega_p(K_{t_{\ast}})}
\left(\frac{\Omega_p^{2+p}(K_{t_{\ast}})}{A^{2-p}(K_{t_{\ast}})}\right)^{\frac{1}{p}} \left(\sigma_M^{\frac{1}{2}-\frac{3p}{2(p+2)}}(t_{\ast})-\sigma_m^{\frac{1}{2}-\frac{3p}{2(p+2)}}(t_{\ast})\right)^2.$$
Here, $\sigma_M^{\frac{1}{2}-\frac{3p}{2(p+2)}}(t_{\ast})$ and $\sigma_m^{\frac{1}{2}-\frac{3p}{2(p+2)}}(t_{\ast})$ are, respectively, the maximum and the minimum of $\sigma^{\frac{1}{2}-\frac{3p}{2(p+2)}}$ on $\partial K_{t_{\ast}}.$
It follows that
\begin{equation}\label{ie: fund}
\sqrt{\frac{2^{\frac{p+2}{p}}\pi^2\Omega_1(K_{t_{\ast}})\Omega_p(K_{t_{\ast}})\varepsilon_p}{d_p\delta}}\geq \left(\frac{\Omega_p^{2+p}(K_{t_{\ast}})}{A^{2-p}(K_{t_{\ast}})}\right)^{\frac{1}{2p}} \left(\sigma_M^{\frac{1}{2}-\frac{3p}{2(p+2)}}(t_{\ast})-\sigma_m^{\frac{1}{2}-\frac{3p}{2(p+2)}}(t_{\ast})\right).
\end{equation}
To bound $\Omega_1(K_{t_{\ast}})\Omega_p(K_{t_{\ast}})$ from above we need to consider two cases. Let $1<p\leq2$. By Lemma \ref{lem: ev of A} we have $A(K_{t_{\ast}})\leq A(K)=\pi.$ Therefore, by the affine isoperimetric inequality and the $p$-affine isoperimetric inequality we infer that
$$\Omega_1(K_{t_{\ast}})\leq2\pi^{\frac{2}{3}}A^{\frac{1}{3}}(K_{t_{\ast}})\leq2\pi,$$
$$\Omega_p(K_{t_{\ast}})\leq2\pi^{\frac{2p}{p+2}}A^{\frac{2-p}{p+2}}(K_{t_{\ast}})\leq2\pi,$$
and thus $\Omega_1(K_{t_{\ast}})\Omega_p(K_{t_{\ast}})\leq 4\pi^2.$\\
Now we proceed to deal with the case $p>2.$ Recall from the evolution equation of the area, Lemma \ref{lem: ev of A}, that
$$\frac{d}{dt}A(K_t)=-\Omega_1(K_t)\geq -2\pi,$$
hence
$$A(K_{\delta})\geq A(K)-2\pi\delta=\pi(1-2\delta).$$
If $\delta<\frac{1}{4}$ then $A(K_{\delta})>\frac{\pi}{2}$. In particular, this yields that $A(K_{t_{\ast}})>\frac{\pi}{2}.$ This observation combined with the $p$-affine isoperimetric inequality imply that
$$\Omega_p(K_{t_{\ast}})\leq2\pi^{\frac{2p}{p+2}}A^{\frac{2-p}{p+2}}(K_{t_{\ast}})\leq2^{\frac{2p}{p+2}}\pi.$$
As $\Omega_1(K_{t_{\ast}})\leq 2\pi$ we get $\Omega_1(K_{t_{\ast}})\Omega_p(K_{t_{\ast}})\leq 2^{\frac{2p+2}{p+2}}\pi^2<4\pi^2$. Consequently,
assuming $\delta<\frac{1}{4}$ together with inequalities (\ref{ie: assum}) and (\ref{ie: fund}) yield
$$\left(\sigma_M^{\frac{1}{2}-\frac{3p}{2(p+2)}}(t_{\ast})-\sigma_m^{\frac{1}{2}-\frac{3p}{2(p+2)}}(t_{\ast})\right)\leq \frac{2\sqrt{2}\pi}{\sqrt{d_p}}\sqrt{\frac{\varepsilon_p}{\delta}}.$$

Define $d_p':=\frac{2\sqrt{2}\pi}{\sqrt{d_p}}.$ Multiplying $K_{t_{\ast}}$ by a factor $\lambda$, depending on $\delta$, where $\lambda\geq1$ so that $A(\lambda K_{t_{\ast}})=\pi$. Note that $\lim\limits_{\delta\to0}\lambda=1.$ In particular, by this assumption and Lemma \ref{lem: minmax of affine} we have
$$1\in\left[\lambda^{\frac{(1-p)}{3(p+2)}}\sigma_m^{\frac{1-p}{2+p}}(t_{\ast}),
\lambda^{\frac{(1-p)}{3(p+2)}}\sigma_M^{\frac{1-p}{2+p}}(t_{\ast})\right].$$
As a result,
$$\lambda^{\frac{4(1-p)}{3(p+2)}}\sigma_M^{\frac{1-p}{2+p}}(t_{\ast})\leq d'_p\sqrt{\frac{\varepsilon_p}{\delta}}+1,$$
and
$$\lambda^{\frac{4(1-p)}{3(p+2)}}\sigma_m^{\frac{1-p}{2+p}}(t_{\ast})\geq 1-d'_p\sqrt{\frac{\varepsilon_p}{\delta}}.$$
Let us assume for now that
\begin{align}\label{e: assumption}
1-d'_p\sqrt{\frac{\varepsilon_p}{\delta}}>0.
\end{align}
Consequently,
 \begin{equation*}
 \frac{1}{\left(1+d'_p\sqrt{\frac{\varepsilon_p}{\delta}}\right)^{\frac{p+2}{p-1}}}\leq\lambda^{\frac{4}{3}}\sigma(t_{\ast})\leq \frac{1}{\left(1-d'_p\sqrt{\frac{\varepsilon_p}{\delta}}\right)^{\frac{p+2}{p-1}}}.
 \end{equation*}
From the last inequality and Lemma \ref{lem: ellipsoid app} we deduce that there exist two ellipses, denoted by $\mathcal{E}_{in}$ and $\mathcal{E}_{out}$ , such that
\begin{equation}\label{e: squeezed}
\mathcal{E}_{in}\subseteq K_{t_{\ast}}\subseteq \mathcal{E}_{out},
\end{equation}
and
$$\left(\frac{A(\mathcal{E}_{out})}{\pi}\right)^{2/3}=\frac{\lambda^{-\frac{4}{3}}}{\left(1-d'_p\sqrt{\frac{\varepsilon_p}{\delta}}\right)^{\frac{p+2}{p-1}}},~
\left(\frac{A(\mathcal{E}_{in})}{\pi}\right)^{2/3}=\frac{\lambda^{-\frac{4}{3}}}{\left(1+d'_p\sqrt{\frac{\varepsilon_p}{\delta}}\right)^{\frac{p+2}{p-1}}}.$$
On the other hand, let us assume that $\delta<\frac{3}{4}c_1^{\frac{4}{3}}$, then by Theorem \ref{thm: Ben}
\begin{equation}\label{e: distance from the intial data}
K_{t_{\ast}}\subseteq K\subseteq K_{t_{\ast}}+\left(\frac{4}{3}\right)^{\frac{3}{4}}\frac{c_2}{c_1}t_{\ast}^{\frac{3}{4}}B_1\subseteq K_{t_{\ast}}+\left(\frac{4}{3}\right)^{\frac{3}{4}}\frac{c_2}{c_1}\delta^{\frac{3}{4}}B_1.
\end{equation}
Combining relations (\ref{e: squeezed}) and (\ref{e: distance from the intial data}) we find
$$\mathcal{E}_{in}\subseteq K\subseteq\mathcal{E}_{out}+\left(\frac{4}{3}\right)^{\frac{3}{4}}\frac{c_2}{c_1}\delta^{\frac{3}{4}}B_1.$$
Set $\delta:=\varepsilon_p^{\frac{\beta}{2+\beta}}$, for a positive $\beta$, in the previous inequality. For $\varepsilon_p<\left(\frac{3}{4}c_1^{\frac{4}{3}}\right)^{\frac{2+\beta}{\beta}},$ we have
\begin{equation}\label{e: final step}
\mathcal{E}_{in}\subseteq K\subseteq\mathcal{E}_{out}+\left(\frac{4}{3}\right)^{\frac{3}{4}}\frac{c_2}{c_1}\varepsilon_p^{\frac{3\beta}{4(2+\beta)}}B_1,
\end{equation}
and
\begin{equation}\label{e: ellipses relations}
\left(\frac{A(\mathcal{E}_{out})}{\pi}\right)^{2/3}=\frac{\lambda^{-\frac{4}{3}}}{\left(1-d'_p\varepsilon_p^{\frac{1}{2+\beta}}\right)^{\frac{p+2}{p-1}}},~
\left(\frac{A(\mathcal{E}_{in})}{\pi}\right)^{2/3}=\frac{\lambda^{-\frac{4}{3}}}{\left(1+d'_p\varepsilon_p^{\frac{1}{2+\beta}}\right)^{\frac{p+2}{p-1}}}.
\end{equation}
We now get back to the assumption (\ref{e: assumption}). If we choose $\varepsilon_p<\left(\frac{1}{d'_p}\right)^{2+\beta}$
then
$$1-d'_p\varepsilon_p^{\frac{1}{2+\beta}}>0.$$
On the other hand, $\delta:=\varepsilon_p^{\frac{\beta}{2+\beta}}<\frac{1}{4}$ and $\varepsilon_p<\left(\frac{3}{4}c_1^{\frac{4}{3}}\right)^{\frac{2+\beta}{\beta}}.$
Therefore, choosing
$$\varepsilon_p<\min\left\{\left(\frac{1}{4}\right)^{\frac{1+\beta}{\beta}},
\left(\frac{1}{d'_p}\right)^{2+\beta},\left(\frac{3}{4}c_1^{\frac{4}{3}}\right)^{\frac{2+\beta}{\beta}}\right\}$$ guarantees that both assumptions (\ref{e: assumption}) and (\ref{e: distance from the intial data}) hold.

Recall that $B_{c_1}\subseteq K$. Therefore, by Lemma \ref{prop: containment principle}, $B_{c_1/2}\subseteq K_t$ for $t\in[0,\eta],$ for an $\eta$ independent of $K$. Precisely, $\eta=\frac{3}{4}c_1^{\frac{4}{3}}\left(1-\left(\frac{1}{2}\right)^{\frac{4}{3}}\right)$ is the time that $B_{c_1}$ shrinks to $B_{c_1/2}$ under the affine normal flow. If we choose $\delta=\varepsilon_p^{\frac{\beta}{2+\beta}}<\eta,$ then from (\ref{e: squeezed}) we get
$$B_{c_1/2}\subseteq K_{t_{\ast}}\subseteq \mathcal{E}_{out}.$$
 From this we conclude that, if
 $$\varepsilon_p<\min\left\{\left(\frac{1}{4}\right)^{\frac{1+\beta}{\beta}},
\left(\frac{1}{d'_p}\right)^{2+\beta},
\left(\frac{3}{4}c_1^{\frac{4}{3}}\right)^{\frac{2+\beta}{\beta}}\left(1-\left(\frac{1}{2}\right)^{\frac{4}{3}}\right)^{\frac{(2+\beta)}{\beta}}
\right\}$$
 then,
$$\left(\frac{4}{3}\right)^{\frac{3}{4}}\frac{c_2}{c_1}\varepsilon_p^{\frac{3\beta}{4(2+\beta)}}B_1= \left(\frac{4}{3}\right)^{\frac{3}{4}}\frac{2c_2}{c_1^2}\varepsilon_p^{\frac{3\beta}{4(2+\beta)}}B_{c_1/2}\subseteq \left(\frac{4}{3}\right)^{\frac{3}{4}}\frac{2c_2}{c_1^2}\varepsilon_p^{\frac{3\beta}{4(2+\beta)}}\mathcal{E}_{out}.$$
By (\ref{e: final step}) we find
\begin{equation}\label{e: final step2}
\mathcal{E}_{in}\subseteq K\subseteq \left(1+\left(\frac{4}{3}\right)^{\frac{3}{4}}\frac{2c_2}{c_1^2}\varepsilon_p^{\frac{3\beta}{4(2+\beta)}}\right)\mathcal{E}_{out}.
\end{equation}
We apply a special linear transformation, $T\in SL(2)$, such that $T\mathcal{E}_{out}$ is a disk. Consequently, by relation (\ref{e: final step2}) we get
\begin{equation}\label{e: final step1}
T\mathcal{E}_{in}\subseteq TK\subseteq \left(1+\left(\frac{4}{3}\right)^{\frac{3}{4}}\frac{2c_2}{c_1^2}\varepsilon_p^{\frac{3\beta}{4(2+\beta)}}\right)T\mathcal{E}_{out}.
\end{equation}
Now from the facts that $\mathcal{E}_{in}\subseteq \mathcal{E}_{out}$, area is invariant under special linear transformations, Lemma \ref{e: ellipsoid app2}, and identities in (\ref{e: ellipses relations}) and we have
\begin{align*}
d_{\mathcal{H}}\left(T\mathcal{E}_{in},\left(1+\left(\frac{4}{3}\right)^{\frac{3}{4}}\frac{2c_2}{c_1^2}\varepsilon_p^{\frac{3\beta}{4(2+\beta)}}\right)
T\mathcal{E}_{out}\right)&\leq \frac{\left(1+\left(\frac{4}{3}\right)^{\frac{3}{4}}\frac{2c_2}{c_1^2}\varepsilon_p^{\frac{3\beta}{4(2+\beta)}}\right)^2A(\mathcal{E}_{out})
-A(\mathcal{E}_{in})}
{\left(1+\left(\frac{4}{3}\right)^{\frac{3}{4}}\frac{2c_2}{c_1^2}\varepsilon_p^{\frac{3\beta}{4(2+\beta)}}\right)\pi\left(\frac{A(\mathcal{E}_{out})}{\pi}\right)^{\frac{1}{2}}}.
\end{align*}
Therefore $d_{\mathcal{H}}\left(T\mathcal{E}_{in},\left(1+\left(\frac{4}{3}\right)^{\frac{3}{4}}\frac{2c_2}{c_1^2}\varepsilon_p^{\frac{3\beta}{4(2+\beta)}}\right)
T\mathcal{E}_{out}\right)$ is bounded by
\begin{align*}
&\frac{1}{\lambda}\frac{\left(1-d'_p\varepsilon_p^{\frac{1}{2+\beta}}\right)^{\frac{3(p+2)}{4(p-1)}}}{\left(1+\left(\frac{4}{3}\right)^{\frac{3}{4}}\frac{2c_2}{c_1^2}
\varepsilon_p^{\frac{3\beta}{4(2+\beta)}}\right)}
\left(\frac{\left(1+\left(\frac{4}{3}\right)^{\frac{3}{4}}\frac{2c_2}{c_1^2}\varepsilon_p^{\frac{3\beta}{4(2+\beta)}}\right)^2}{\left(1-d'_p
\varepsilon_p^{\frac{1}{2+\beta}}\right)^{\frac{3(p+2)}{2(p-1)}}}
-\frac{1}{\left(1+d'_p\varepsilon_p^{\frac{1}{2+\beta}}\right)^{\frac{3(p+2)}{2(p-1)}}}\right)\\
&\leq
\frac{\left(1+\left(\frac{4}{3}\right)^{\frac{3}{4}}\frac{2c_2}{c_1^2}\varepsilon_p^{\frac{3\beta}{4(2+\beta)}}\right)^2}{\left(1-d'_p
\varepsilon_p^{\frac{1}{2+\beta}}\right)^{\frac{3(p+2)}{2(p-1)}}}
-\frac{1}{\left(1+d'_p\varepsilon_p^{\frac{1}{2+\beta}}\right)^{\frac{3(p+2)}{2(p-1)}}}.
\end{align*}

To optimize the Hausdorff distance we set $\beta=\frac{4}{3}.$ Observe that
\begin{align*}
\lim_{\varepsilon_p\to0^{+}}
\frac{\frac{\left(1+\left(\frac{4}{3}\right)^{\frac{3}{4}}\frac{2c_2}{c_1^2}\varepsilon_p^{\frac{3}{10}}\right)^2}{\left(1-d_p'
\varepsilon_p^{\frac{3}{10}}\right)^{\frac{3(p+2)}{2(p-1)}}}
-\frac{1}{\left(1+d_p'\varepsilon_p^{\frac{3}{10}}\right)^{\frac{3(p+2)}{2(p-1)}}}}{\varepsilon_p^{\frac{3}{10}}}&=
2\left(1+\left(\frac{4}{3}\right)^{\frac{3}{4}}\frac{2c_2}{c_1^2}+\frac{3(p+2)}{2(p-1)}d_p'\right).
\end{align*}
Define $\mathcal{E}:=T\mathcal{E}_{in}$, $\mathcal{D}:=T\mathcal{E}_{out}$ and $C_p:=3\left(1+\left(\frac{4}{3}\right)^{\frac{3}{4}}\frac{2c_2}{c_1^2}+\frac{3(p+2)}{2(p-1)}d_p'\right).$ Therefore, choosing $\varepsilon_p$ small enough implies the claim for $p>1.$

To complete the proof of the main theorem we need to address the case $p=1$. We note that if
\begin{equation*}
\frac{\Omega_1^{3}(K)}{8A(K)}>\pi^2(1-\varepsilon),
\end{equation*}
then for every $p>1$
\begin{equation*}
\left(\frac{\Omega_p^{2+p}(K)}{2^{2+p}A^{2-p}(K)}\right)^{\frac{1}{p}}>\pi^2(1-\varepsilon).
\end{equation*}
This is because the function $p\mapsto \left(\frac{\Omega_p^{2+p}(K)}{2^{2+p}A^{2-p}(K)}\right)^{\frac{1}{p}}$ is increasing, \cite{Lutwak2}. Hence, to prove the stability of the affine isoperimetric inequality we can continue the argument for the stability of the $p$-affine isoperimetric inequality, for example with $p=2.$ In this case $C_1=C_2.$ The proof is now complete.
\end{proof}
\textbf{Acknowledgment.}
I would like to thank Alina Stancu and K\'{a}roly B\"{o}r\"{o}czky for comments and suggestions that have improved the initial manuscript.
I am indebted to two referees for the very careful reading of the original submission.
\bibliographystyle{amsplain}

\end{document}